\documentclass[11pt,a4paper]{article}
\usepackage{graphicx,amssymb,amsmath}
\usepackage{epsfig}
\usepackage{epsf}
\usepackage{amsfonts}
\usepackage{latexsym}  
\usepackage{bbm}
\usepackage{float}

\newtheorem{theorem}{Theorem}
\newtheorem{proposition}{Proposition}

\newcommand{\later}[1]{{}}
\newcommand{\old}[1]{{}}
\long\def\ignore#1{}

\newcommand{\EE}{\mathbb{E}} 
\newcommand{\RR}{\mathbb{R}} 

\def\I{{\mathcal I}}

\def\Q{{\mathcal Q}}

\def\eps{\varepsilon}
\def\rho{\varrho}

\newcommand{\vol}{{\rm Vol}}

\title{Approximate Euclidean Ramsey theorems}

\author{
Adrian Dumitrescu\thanks{Department of Computer Science,
University of Wisconsin--Milwaukee, Email:~\texttt{ad@cs.uwm.edu}.
Supported in part by NSF CAREER grant CCF-0444188.
}}

\begin{document}

\maketitle

\begin{abstract}
According to a classical result of Szemer\'{e}di,
every dense subset of $1,2,\ldots,N$ contains an 
arbitrary long arithmetic progression, if $N$ is large enough.
Its analogue in higher dimensions due to F\"urstenberg and Katznelson
says that every dense subset of $\{1,2,\ldots,N\}^d$ contains an
arbitrary large grid, if $N$ is large enough. 
Here we generalize these results for separated point sets    
on the line and respectively in the Euclidean space:
(i) every dense separated set of points in some interval
$[0,L]$ on the line contains an arbitrary long approximate arithmetic
progression, if $L$ is large enough. 
(ii) every dense separated set of points in the $d$-dimensional cube $[0,L]^d$
in $\RR^d$ contains an arbitrary large approximate grid, if $L$ is
large enough.  A further generalization for any finite pattern in $\RR^d$ 
is also established. The separation condition is shown to be necessary
for such results to hold. In the end we show that every sufficiently
large point set in $\RR^d$ contains an arbitrarily large subset of
almost collinear points. No separation condition is needed in this case. 
\end{abstract}


\section{Introduction} \label{sec:intro}

Let us start by recalling the classical result of Ramsey from 1930:

\begin{theorem} {\rm(Ramsey \cite{R30}).} \label{T-R} 
Let $p \leq q$, and $r$ be positive integers. Then there exists a positive integer
$N=N(p,q,r)$ with the following property: If $X$ is a set with $N$
elements, for any $r$-coloring of the $p$-element subsets of $X$,
there exists a subset $Y$ of $X$ with at least $q$ elements such that 
all $p$-element subsets of $Y$ have the same color.
\end{theorem} 

As noted in \cite{EG+73}, perhaps the first Ramsey type result of a
geometric nature is Van der Waerden's theorem on arithmetic progressions:

\begin{theorem} {\rm(Van der Waerden \cite{W27}).} \label{T-W} 
For every positive integers $k$ and $r$, there exists 
a positive integer $W=W(k,r)$ with the following property: 
For every $r$-coloring of the integers $1,2,\ldots,W$ 
there is a monochromatic arithmetic progression of $k$ terms.
\end{theorem} 

As early as 1936,  Erd\H{o}s and Tur\'an have suggested that a stronger
{\em density} statement must hold. Only in 1975, Szemer\'{e}di
succeeded to confirm this belief with his celebrated result:

\begin{theorem} {\rm (Szemer\'{e}di \cite{S75}).} \label{T-S} 
For every positive integer $k$ and every $c>0$, there exists $N=N(k,c)$ 
such that every subset $X$ of $\{1,2,\ldots,N\}$ of size at least
$c N$ contains an arithmetic progression with $k$ terms.
\end{theorem} 

This is a fundamental result with relations to many areas in mathematics. 
Szemer\'{e}di's proof is very complicated and is regarded as a 
mathematical tour de force in combinatorial reasoning \cite{GRS90,N95}. 
Another proof of this result was obtained by means of ergodic theory by  
F\"urstenberg~\cite{F77} in 1977. 

A homothetic copy of $\{1,2,\ldots,k\}^d$ is also called a $k$-{\em
grid} in $\RR^d$. The following generalization of Van der Waerden's
theorem to higher dimensions is given by the Gallai--Witt theorem~\cite{GRS90,N95}: 

\begin{theorem} {\rm (Gallai--Witt \cite{N95}).} \label{T-GW} 
For every positive integers $d$, $k$ and $r$, there exists a 
positive integer $N=N(d,k,r)$ with the following property: 
For every $r$-coloring of the integer lattice points in 
$\{1,2,\ldots,N\}^d$, there exists a monochromatic homothetic copy of 
$\{1,2,\ldots,k\}^d$. More precisely, there exist 
$(a_1,a_2,\ldots,a_d) \in \{1,2,\ldots,N\}^d$, and a positive
integer $x$ such that all points of the form 
$$ (a_1+i_1 x, a_2+i_2 x, \ldots, a_d+i_d x), 
\ \ \ i_1,i_2,\ldots,i_d \in \{0,1,\ldots,k-1\} $$
are of the same color.
\end{theorem} 

A higher dimensional generalization of Szemer\'{e}di's density
theorem was obtained by F\"urstenberg and Katznelson \cite{FK78}; see
also \cite{N95}. 

\begin{theorem} {\rm (F\"urstenberg--Katznelson \cite{FK78}).} \label{T-FK} 
For every positive integers $d$, $k$ and  every $c>0$, there exists
a positive integer $N=N(d,k,c)$  with the following property: 
every subset $X$ of $\{1,2,\ldots,N\}^d$ of size at least
$c N^d$ contains a homothetic copy of $\{1,2,\ldots,k\}^d$.
\end{theorem} 

The proof of F\"urstenberg and Katznelson uses infinitary methods in
ergodic theory. As noted in \cite{N95}, no combinatorial proof is
known.  

\old{
Many other Ramsey type problems in the Euclidean space 
have been investigated in a series of papers 
by Erd\H{o}s et~al.\ \cite{EG+73,EG+75a,EG+75b} in the early 1970s,
and by Graham \cite{G90,G94}. See also Ch.~6.3 in the problem collection by 
Bra\ss, Moser and Pach \cite{BMP05}.
} 

In the first part of our paper (Section~\ref{S-approximate}), we
present analogues of Theorems~\ref{T-W},~\ref{T-S},~\ref{T-GW}, and~\ref{T-FK}, 
for point sets in the Euclidean space. Specifically, we obtain
(restricted) Ramsey theorems for {\em separated} point sets, for finding
approximate homothetic copies of an arithmetic progression on the line
and respectively of a grid in $\RR^d$.  
The latter result carries over for any finite pattern point set and
every dense and sufficiently large separated point set in $\RR^d$. 
It is worth noting that the separation condition is necessary 
for such results to hold (Proposition~\ref{P1} in Section~\ref{S-approximate}).  
While for Theorems~\ref{T-W},~\ref{T-S},~\ref{T-GW}, and~\ref{T-FK},
the separation condition comes for free for any set of integers, it
has to be explicitly enforced for point sets.

The exact statements of our results (Theorems~\ref{T1},~\ref{T2} and
\ref{T3}) are to be found in Section~\ref{S-approximate} following the
definitions. Fortunately, the proofs of these theorems are much
simpler than of their exact counterparts previously
mentioned. Moreover, the resulting upper bounds are much better
than those one would get from the integer theorems. The proofs are
constructive and yield very simple and efficient algorithms for
computing the respective approximate homothetic copies given input
point sets satisfying the requirements.   

In the second part (Section~\ref{S-collinear}), we present an
unrestricted theorem (Theorem~\ref{T4}) which shows the existence of 
an arbitrary large subset of almost collinear points in every
sufficiently large point set in $\RR^d$. No separation condition is 
needed in this result. 

\vspace{-0.5\baselineskip}
\paragraph {Applications.} 
Many other Ramsey type problems in the Euclidean space 
have been investigated in a series of papers 
by Erd\H{o}s et~al.\ \cite{EG+73,EG+75a,EG+75b} in the early 1970s,
and later by Graham \cite{G80,G83,G85,G90,G94}. 
Van der Waerden's theorem on arithmetic progressions has inspired 
new connections and numerous results in number theory, combinatorics,
and combinatorial
geometry~\cite{BP05,DJ10,EG79,G80,G04,G07,G08,GRS90,GS06,HJ63,N95,S94},
where we only named a few here.

Our analogues of Theorems~\ref{T-W},~\ref{T-S},~\ref{T-GW}, and~\ref{T-FK},  
for point sets in the Euclidean space 
may also find fruitful applications in combinatorial and computational geometry.
It is obvious that general point sets are much more common in these
areas than the rather special integer or lattice point sets that occur
in number theory and integer combinatorics.
A first application needs to be mentioned: A result similar
to our Theorem~\ref{T1} has been proved instrumental in 
settling a conjecture of Mitchell~\cite{M07} on
illumination for maximal unit disk packings: It is shown~\cite{DJ10}
that any dense (circular) forest with congruent unit trees that is
deep enough has a hidden point. The result that is needed there
is an approximate equidistribution lemma for separated points on the
line, which is a relaxed version of our Theorem~\ref{T1}.

\section{Approximate homothetic copies of any pattern} \label {S-approximate}

\paragraph {Definitions.} Let $\delta>0$. A point set $S$ in $\RR^d$
is said to be $\delta$-{\em separated} if the minimum pairwise distance
among points in $S$ is at least $\delta$.
For two points $p,q \in \RR^d$, let $d(p,q)$ denote the Euclidean
distance between them. 
The closed ball of radius $r$ in $\RR^d$ centered at point $z=(z_1,\ldots,z_d)$ is
$$ B_d(z,r) = \{ x \in \RR^d \ | \ d(z,x) \leq r \}=
\{(x_1,\ldots,x_d) \ | \ \sum_{i=1}^d (x_i -z_i)^2 \leq r^2 \}. $$

Given a point set (or ``pattern'')
$P=\{p_1,\ldots,p_k\}$ of $k$ points in $\RR^d$ and another point set
$Q$ with $k$ points: (i) $Q$ is {\em similar} to $P$, if it is a
magnified/shrunk and possibly rotated copy of $P$.
(ii) $Q$ is {\em homothetic} to $P$, if it is a magnified/shrunk copy
of $P$ in the same position (with no rotations). 

Approximate similar copies and approximate homothetic copies are defined
as follows. See also Fig.~\ref{f1} for an illustration.
Given point sets $P$ and $Q$ as above and $0 <\eps \leq 1/3$:
\begin{itemize}
\item $Q$ is an $\eps$-{\em approximate similar} copy of of $P$, if 
there exists $Q'$ so that $Q'$ is similar to $P$, and each point 
$q'_i \in Q'$ contains a (distinct) point $q_i \in Q$ in the ball of radius 
$\eps d $ centered at $q'_i$, where $d$ is the minimum
pairwise distance among points in $Q'$. 
\item $Q$ is an $\eps$-{\em approximate homothetic} copy of of $P$, if 
there exists $Q'$ so that $Q'$ is homothetic to $P$, and
each point $q'_i \in Q'$ contains a (distinct) point $q_i \in Q$ in
the ball of radius $\eps d$ centered at $q'_i$, where $d$ is the
minimum pairwise distance among points in $Q'$.  
\end{itemize}

\begin{figure} [htb]
\centerline{\epsfxsize=5in \epsffile{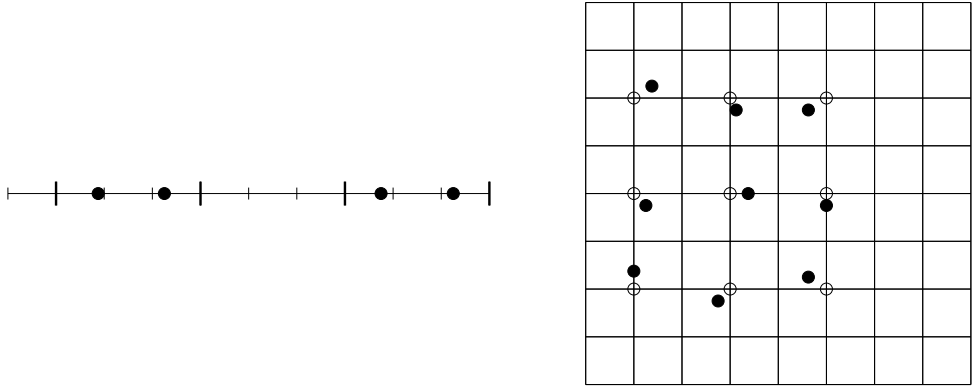}}
\caption{\small Left: a $4$-term arithmetic progression (thick vertical
bars) and a $1/3$-approximate $4$-term arithmetic progression (filled
circles) on the line. 
Right: a $3$-grid (empty circles) and a $1/4$-approximate $3$-grid
(filled circles) in $\RR^2$.}
\label{f1}
\end{figure}

The condition $\eps \leq 1/3$ is imposed to ensure
that any two balls of radius $\eps d$ around points in $Q'$ are
disjoint, and moreover, that any two distinct points of $Q$ are
separated by a constant times $d$, in this case by at least $d/3$. 
In our theorems, $\eps$-approximate means $\eps$-approximate homothetic
copy. We start with $\eps$-approximate arithmetic progressions on the
line by proving the following analogue of Theorem~\ref{T-S} for points
on the line: 

\begin{theorem} \label{T1}
For every positive integer $k$, $c,\delta>0$, and $0<\eps\leq 1/3$, 
there exists a positive number $Z_0=Z_0(k,c,\delta,\eps)$  with the
following property:  Let $S$ be a $\delta$-separated point set in an 
interval $I$ of length $|I|=L$ with at least $cL$ points, where $L
\geq Z_0$. Then $S$ contains a $k$-point subset that forms an $\eps$-approximate 
arithmetic progression of $k$ terms. Moreover, one can set
$$ Z_0(k,c,\delta,\eps) = 2 \delta \cdot (ks)^j, {\rm \ \ where \ \ }
s=\left\lceil \frac{1}{\eps} \right\rceil, \ \  
r =\frac{k}{k-1}, \ \ 
j = \left \lceil \frac{\log \frac{2}{c \delta}} {\log r} \right \rceil. 
$$
\end{theorem}
\begin{proof} 
Without loss of generality, $I=[0,L]$. 
Put $s=\lceil \frac{1}{\eps} \rceil $. Conduct an iterative process as follows.
In step 0: Let $I_0=I$, and subdivide the interval $I_0$ into $ks$
half-closed intervals\footnote{When subdividing a closed interval, the
first $k-1$ resulting sub-intervals are half-closed, and the $k$th
sub-interval is closed. 
When subdividing a half-closed interval, all resulting sub-intervals
are half-closed.} of equal length. Let $x=L/(ks)$ be the common length
of the sub-intervals. For $t=0,\ldots,s-1$ consider the
system $\I_t$ of $k$ disjoint sub-intervals with left endpoints of
coordinates $tx,(t+s)x,(t+2s)x,\ldots,(t+(k-1)s)x$. 
Observe that the $s$ systems of intervals $\I_t$ partition the interval $I_0$.
If for some $t$ , $0 \leq t \leq s-1$, each of the $k$ intervals
contains at least one point in $S$, stop. 
Otherwise in each of the $s$ systems of $k$ intervals, 
at least one of the $k$ intervals is empty, so all the points are
contained in at most $ks-s$ intervals from the total of $ks$. 
Now pick one of the remaining $(k-1)s$ intervals, which contains the
most points of $S$, say $I_1$. 
In step $i$, $i \geq 1$: Subdivide $I_i$ into $ks$ half-closed intervals
of equal length and proceed as before. 

In the current step $i$, the process either (i) terminates 
successfully by finding an interval $I_i$ subdivided into $ks$
sub-intervals making  $s$ systems of intervals, and in at least one of
the systems, each sub-interval contains at least one point in $S$, or
(ii) it continues with another subdivision in step $i+1$. We show that
if $L$ is large enough, and the number of subdivision steps is large
enough, the iterative process terminates successfully. 

Let $L_0=|I|=L$ be the initial interval length, and $m_0 \geq c|I|$ be the
(initial) number of points in $I_0$.  
At step $i$, $i \geq 0$, let $m_i$ be the number of points in $I_i$, and 
let $L_i =|I_i|$ be the length of interval $I_i$. Clearly
\begin{equation} \label{E13}
L_i= \frac{L}{(ks)^i}, {\rm \ \ and \ \ } 
m_i \geq \frac{m_0}{(k-1)^i s^i} \geq \frac{c L}{(k-1)^i s^i} . 
\end{equation}
Let $j$ be a positive integer so that
\begin{equation} \label{E11}
c \cdot \delta \cdot \left (\frac{k}{k-1} \right)^j \geq 2, 
{\rm \ \ e.g., set \ \ } 
j = \left \lceil \frac{\log \frac{2}{c \delta}} {\log r} \right \rceil, 
{\rm \ \ where \ \ } r =\frac{k}{k-1}. 
\end{equation}
Now set $Z_0(k,c,\delta,\eps) = 2 \delta \cdot (ks)^j$. 
If $L \geq Z_0$, as assumed, then by our choice of parameters we have
\begin{equation} \label{E1}
L_j = \frac{L}{(ks)^j} \geq \frac{Z_0}{(ks)^j} =
\frac{2 \delta \cdot (ks)^j}{(ks)^j} =2\delta, 
\end{equation}
and
\begin{equation} \label{E14}
m_j \cdot \delta \geq \frac{c L \cdot \delta}{(k-1)^j s^j} =
c \cdot \delta \cdot \left (\frac{k}{k-1} \right)^j \frac{L}{(ks)^j} 
\geq \frac{2L}{(ks)^j}= 2 L_j. 
\end{equation}
Since the point set is $\delta$-separated, an interval packing argument 
on the line using~\eqref{E1} gives 
\begin{equation} \label{E15}
m_j \delta \leq L_j+\frac{\delta}{2}+\frac{\delta}{2} = L_j+\delta
\leq \frac{3}{2} L_j. 
\end{equation}
Observe that~\eqref{E15} is in contradiction to~\eqref{E14},  
which means that the iterative process cannot reach step $j$.
We conclude that for some $0 \leq i \leq j-1$, step $i$ is successful:
we found a system of $k$ intervals of length $x$ with left endpoints at
coordinates $a_0=tx,a_1=(t+s)x,a_2=(t+2s)x,\ldots,a_{k-1}=(t+(k-1)s)x$, 
each containing a distinct point, say $b_p \in S$, $p=0,1,\ldots,k-1$. 
Observe that the $k$ points $\{a_p : p=0,1,\ldots,k-1\}$ form an (exact) 
arithmetic progression of $k$ terms with common difference equal to $sx$. 
It is now easy to verify that the $k$ points $b_p$ form an $\eps$-approximate
arithmetic progression of $k$ terms, since for $p=0,1,\ldots,k-1$
$$ a_p \leq b_p \leq a_p +x \ \ {\rm and } \ \ \eps s \geq 1, 
\ \ {\rm thus } \ \ x \leq \eps s x \ \ {\rm and } \ \  
b_p \in [a_p,a_p+ \eps sx]. $$ 
This completes the proof.
\end{proof}

The next proposition shows that the separation condition in the
theorem is necessary, for otherwise, even a $3$-term approximate
arithmetic progression cannot be guaranteed, irrespective of the size
of the point set.

\begin{proposition} \label{P1}
For any $n$, and for any $0 \leq \eps <1/3$, 
there exists a set of $n$ points in $[0,1]$, without an
$\eps$-approximate arithmetic progression of $3$ terms.   
\end{proposition}
\begin{proof}
Let $\xi=\frac{1}{3}-\eps$. 
Let $S=\{\xi^i\ | \ i=0,\ldots,n-1\}$. Assume for contradiction that
$\{q_1,q_2,q_3\}$ is an $\eps$-approximate arithmetic progression of
$3$ terms, where $q_1<q_2<q_3$, and $q_1,q_2,q_3 \in S$. Then there
exist $a$ and $r>0$, so that $a-r$, $a$ and $a+r$ form a $3$-term
arithmetic progression, and we have: 
\begin{align*} \label{E4}
a-r-\eps r &\leq q_1 \leq a-r+\eps r, \\
a-\eps r &\leq q_2 \leq a+\eps r, \\
a+r-\eps r &\leq q_3 \leq a+r+\eps r.
\end{align*} 
From the first and the third inequalities we obtain
\begin{equation*} \label{E5}
a-\eps r \leq \frac{q_1+q_3}{2} \leq a+\eps r, 
\end{equation*} 
therefore 
\begin{equation} \label{E6}
\left| \frac{q_1+q_3}{2} -q_2 \right| \leq 2\eps r. 
\end{equation} 

Further note that
\begin{equation*} \label{E7}
q_3-q_1 \geq a+r-\eps r -(a-r +\eps r)= 2(1-\eps)r, 
\end{equation*} 
hence 
\begin{equation*} \label{E8}
r \leq \frac{q_3-q_1}{2(1-\eps)}. 
\end{equation*} 
By substituting this bound into~\eqref{E6}, we have
\begin{equation} \label{E9}
\left| \frac{q_1+q_3}{2} -q_2 \right| \leq \frac{\eps}{1-\eps} \cdot
(q_3-q_1) \leq \frac{\eps}{1-\eps} \cdot q_3. 
\end{equation} 
On the other hand
\begin{equation} \label{E10}
\left| \frac{q_1+q_3}{2} -q_2 \right| \geq \frac{q_1+q_3}{2} - q_2
\geq \frac{q_3}{2} - q_2.
\end{equation} 

Putting inequalities~\eqref{E9} and~\eqref{E10} together and dividing
by $q_3$ yields
$$ \frac{1}{2}- \frac{q_2}{q_3} \leq \frac{\eps}{1-\eps}. $$
Obviously, $\frac{q_2}{q_3} \leq \xi$, hence 
$$ \frac{1}{6} +\eps = \frac{1}{2}- \frac{1}{3} +\eps =
\frac{1}{2}- \xi \leq  \frac{1}{2}- \frac{q_2}{q_3} \leq
\frac{\eps}{1-\eps}. $$ 
Equivalently,
$$ \frac{1}{6} \leq \frac{\eps^2}{1-\eps}, $$ 
which is impossible for $\eps<1/3$. Indeed, the quadratic function
$f(x)=6x^2+x-1$ is strictly negative for $0<x<1/3$.
We have thereby reached a contradiction.
We conclude that $S$ has no $\eps$-approximate arithmetic progression
of $3$ terms.    
\end{proof}

\smallskip
\noindent{\bf Remark.} The following slightly different form of
Proposition~\ref{P1} may be convenient: 
For any $n$ there exists a set of $n$ points in $[0,1]$, without an 
$\eps$-approximate arithmetic progression of $3$ terms, for any 
$0 \leq \eps \leq 1/4$. For the proof, take 
$S=\{1/8^i\ | \ i=0,\ldots,n-1\}$, and proceed in the same way. 

\medskip
For a $d$-dimensional cube $\Pi_{i=1}^d [a_i,b_i]$, let us refer to  
$(a_1,\ldots,a_d)$ as the {\em first vertex} of the $d$-dimensional cube.
We now continue with $\eps$-approximate grids in $\RR^d$ by proving the
following analogue of Theorem~\ref{T-FK} for points in $\RR^d$:

\begin{theorem} \label{T2}
For every positive integers $d,k$, and $c,\delta>0$, and $0<\eps\leq 1/3$, 
there exists a positive number $Z_0=Z_0(d,k,c,\delta,\eps)$  with the
following property:  Let $S$ be a $\delta$-separated point set in the
$d$-dimensional cube $Q=[0,L]^d$, with at least $c L^d$ points, where $L \geq Z_0$.
Then $S$ contains a subset that forms an $\eps$-approximate $k$-grid
in $\RR^d$. Moreover, one can set
$$ Z_0(d,k,c,\delta,\eps) = 2 \delta \cdot (ks)^j, {\rm \ \ where \ \ }
s=\left\lceil \frac{\sqrt{d}}{\eps} \right\rceil, \ \  
r =\frac{k^d}{k^d-1}, \ \ 
j = \left \lceil \frac{\log \frac{\kappa_d}{c \delta}} {\log r} \right \rceil. 
$$
Here $\kappa_d$ (in the expression of $j$) is a constant depending on
$d$: 
\begin{equation}\label{E22}
\kappa_d=\left\lceil \frac{3^d \cdot (d/2)!}{\pi^{d/2}} \right \rceil,
{\rm \ if \ } d {\rm \ is \ even, \ and \ } 
\kappa_d=\left\lceil \frac{3^d \cdot (1\cdot3 \cdots d)}
{2\cdot (2\pi)^{(d-1)/2}} \right \rceil,
{\rm \ if \ } d {\rm \ is \ odd}.
\end{equation}
\end{theorem}
\begin{proof} 
For simplicity of calculations, we first present the proof for $d=2$
by outlining the differences  from the one-dimensional case; the
argument for $d \geq 3$ is analogous, with the specific calculations
in the second part of the proof.  

Recall that we have set $\kappa_2=\lceil\frac{9}{\pi}\rceil=3$. 
Put $s=\lceil \frac{\sqrt2}{\eps} \rceil $. Conduct an iterative process as follows.
In step 0: Let $Q_0=Q$, and subdivide the square $Q_0$ into $(ks)^2$
smaller congruent squares. 
Let $x=L/(ks)$ be the common side length of these squares. For
$t_1,t_2 \in \{0,\ldots,s-1\}$ consider the system $\Q_{t_1,t_2}$ of $k^2$
disjoint squares with first vertices of coordinates 
$(t_1+i_1 s, t_2+i_2 s)x$, where $i_1,i_2 \in \{0,1,\ldots,k-1\}$. 
Observe that the $s^2$ systems of squares $\Q_{t_1,t_2}$ partition the
square $Q_0$. If for some $(t_1,t_2)$ , $0 \leq t_1, t_2 \leq s-1$,
each of the $k^2$ squares in the respective system contains at least
one point in $S$, stop. Otherwise in each of the $s^2$ systems of $k^2$ squares, 
at least one of the $k^2$ squares is empty, so all the points are
contained in at most $k^2 s^2-s^2$ squares from the total of $k^2 s^2$. 
Now pick one of the remaining $s^2(k^2-1)$ squares, which contains the
most points of $S$, say $Q_1$. 
In step $i$, $i \geq 1$: Subdivide $Q_i$ into $(ks)^2$ smaller congruent squares
and proceed as before. 

In the current step $i$, the process either (i) terminates 
successfully by finding a square $Q_i$ subdivided into $(ks)^2$
smaller squares making  $s^2$ systems of squares, and in at least one of
the systems, each smaller square contains at least one point in $S$, or
(ii) it continues with another subdivision in step $i+1$. We show that
similar to the one-dimensional case, if $L$ is large enough, and
the number of subdivision steps is large enough, the iterative process
terminates successfully.  

Let $L_0=L$ be the initial square side of $Q_0$, and $m_0 \geq c L^2$ be the
(initial) number of points in $Q_0$.  
At step $i$, $i \geq 0$, let $m_i$ be the number of points in $Q_i$, and 
let $L_i$ be the side length of $Q_i$. Clearly
$$ L_i= \frac{L}{(ks)^i}, {\rm \ \ and \ \ } 
m_i \geq \frac{m_0}{(k^2-1)^i s^{2i}} \geq \frac{c L^2}{(k^2-1)^i s^{2i}} . $$ 

Let $j$ be a positive integer so that
\begin{equation} \label{E12}
c \cdot \delta^2 \cdot \left (\frac{k^2}{k^2-1} \right)^j \geq \kappa_2=3, 
{\rm \ \ e.g., set \ \ } 
j = \left \lceil \frac{\log \frac{3}{c \delta^2}} {\log r} \right \rceil, 
{\rm \ \ where \ \ } r =\frac{k^2}{k^2-1}. 
\end{equation}
Now set $Z_0(2,k,c,\delta,\eps) = 2 \delta \cdot (ks)^j$.
If $L \geq Z_0$, as assumed, then by our choice of parameters we have
\begin{equation} \label{E2}
L_j = \frac{L}{(ks)^j} \geq \frac{Z_0}{(ks)^j} =
\frac{2 \delta \cdot (ks)^j}{(ks)^j} =2\delta, 
\end{equation}
and
\begin{equation} \label{E3}
m_j \cdot \delta^2 \geq \frac{c L^2 \cdot \delta^2}{(k^2-1)^j s^{2j}} =
c \cdot \delta^2 \cdot \left (\frac{k^2}{k^2-1} \right)^j \frac{L^2}{(ks)^{2j}} 
\geq \frac{3L^2}{(ks)^{2j}}= 3 L_j^2. 
\end{equation}

Note that~\eqref{E2} is identical with~\eqref{E1} from the
one-dimensional case. 
Since $S$ is $\delta$-separated, the disks of radius $\delta/2$
centered at the points of $S$ are interior-disjoint. A straightforward
packing argument yields 
\begin{equation} \label{E16}
m_j \frac{\pi \delta^2}{4} \leq  (L_j +\delta)^2 \leq 
\left(\frac{3}{2} L_j \right)^2= \frac{9}{4} L_j^2, 
\end{equation}
where the last inequality is implied by~\eqref{E2}.
Inequality~\eqref{E16} is equivalent to
\begin{equation} \label{E17}
m_j \cdot \delta^2 \leq \frac{9}{\pi} L_j^2. 
\end{equation}
However this is contradiction with inequality~\eqref{E3}
(by the setting $\kappa_2=\lceil\frac{9}{\pi}\rceil=3$). 
This means that the iterative process cannot reach step $j$.

We conclude that for some $0 \leq i \leq j-1$, step $i$ is successful:
we found a system of $k^2$ disjoint squares of side $x$ 
with first vertices $a_{i_1,i_2} =(t_1+i_1 s, t_2+i_2 s)x$, 
where $i_1,i_2 \in \{0,1,\ldots,k-1\}$,
each containing a distinct point, say $b_{i_1,i_2} \in S$, 
for $i_1,i_2 \in \{0,1,\ldots,k-1\}$.
Observe that the $k^2$ points $a_{i_1,i_2}$ form an (exact) 
grid $Q'$ of $k^2$ points with side length equal to $sx$. 
As in the one-dimensional case, it is now easy to verify that the
$k^2$ points $b_{i_1,i_2}$ form an $\eps$-approximate grid of $k^2$ points,
since for $i_1,i_2 \in \{0,1,\ldots,k-1\}$
\begin{equation} \label{E23}
\eps s \geq \sqrt2, \ \ {\rm thus } \ \ d(a_{i_1,i_2},b_{i_1,i_2})
\leq x \sqrt2 \leq \eps x s. 
\end{equation}
Note that the minimum distance among the points in $Q'$ is $sx$, and
this completes the proof for the planar case ($d=2$).

The argument for the general case $d \geq 3$ is analogous and the
calculations in deriving the upper bound are as follows. 
The inequality~\eqref{E2} remains valid. By the choice of parameters
$r$ and $j$, we have
\begin{equation} \label{E20}
c \cdot \delta^d \cdot \left (\frac{k^d}{k^d-1} \right)^j \geq \kappa_d.
\end{equation}
The analogue of~\eqref{E3} is
\begin{equation} \label{E18}
m_j \cdot \delta^d \geq \frac{c L^d \cdot \delta^d} {(k^d -1)^j s^{dj}} =
c \cdot \delta^d \cdot \left(\frac{k^d}{k^d-1} \right)^j \frac{L^d}{(ks)^{dj}} 
\geq \kappa_d \cdot \frac{L^d}{(ks)^{dj}}= 
\kappa_d \cdot \left(\frac{L}{(ks)^j}\right)^d= \kappa_d \cdot L_j^d. 
\end{equation}
The packing argument in $\RR^d$ yields
\begin{equation} \label{E19}
m_j \cdot \vol_d \left(\frac{\delta}{2}\right) \leq
\left(\frac{3}{2}\right)^d L_j^d,
\end{equation}
where $\vol_d(r)$ is the volume of the sphere of radius $r$ in
$\RR^d$. It is well-known that
\begin{equation} \label{E21}
\vol_d(r)= \begin{cases}
\dfrac{\pi^{d/2}}{(d/2)!} \cdot r^d & {\rm if \ } d \ {\rm is \ even},
\medskip \\
\dfrac{2 \cdot (2\pi)^{(d-1)/2}}{1 \cdot 3 \cdots d} \cdot r^d & {\rm if \ } 
d \ {\rm is \ odd}.
\end{cases}
\end{equation}
To obtain a contradiction in the argument, as in the previous cases,
one sets $\kappa_d$ as in~\eqref{E22} taking into account
\eqref{E21}. The setting of $s$ is such that the analogue of
\eqref{E23} is ensured. This completes the proof of Theorem~\ref{T2}. 
\end{proof}

\medskip
By selecting a sufficiently fine grid in Theorem~\ref{T2}, one obtains
by similar means the following general statement for any pattern in $\RR^d$:

\begin{theorem} \label{T3}
For every positive integer $d$, finite pattern $P \subset \RR^d$, $|P|=k$, and
$c,\delta>0$, and $0<\eps\leq 1/3$,  
there exists a positive number $Z_0=Z_0(d,P,c,\delta,\eps)$  with the
following property:  Let $S$ be a $\delta$-separated point set in the
$d$-dimensional cube $Q=[0,L]^d$, with at least $c L^d$ points, where $L \geq Z_0$.
Then $S$ contains a subset that is an $\eps$-approximate homothetic
copy of $P$.
\end{theorem}

Observe that the iterative procedures used in the proofs of Theorems
\ref{T1},~\ref{T2} and~\ref{T3}, yield very simple and efficient 
algorithms for computing the respective approximate homothetic copies
given input point sets satisfying the imposed requirements.
For instance in Theorems~\ref{T1} and~\ref{T2}, the number of
iterations, $j$, is given by~\eqref{E11} and resp.~\eqref{E12}, 
and each iteration takes linear time (in the number of points).

\vspace{-0.5\baselineskip}
\paragraph {Remark.} The following connection between our result
Theorem~\ref{T1} and Szemer\'{e}di's Theorem~\ref{T-S} is worth 
making. If one makes abstraction of the bounds obtained, 
the qualitative statement in Theorem~\ref{T1} can be obtained as a
corollary from Theorem~\ref{T-S}. Here is a proof.
For simplicity let $n=L$ be integer. 
Take any set of $cn$ points. Since the set is $\delta$-separated every
interval $[i,i+1]$ has at most $1/\delta$ points. 
Therefore, there are at least $c\delta n$ intervals with at least one
point. By Theorem~\ref{T-S}, we know that if $n$ is large enough then
we can find $k/\eps$ intervals which form an arithmetic progression of
length $k/\eps$ (just think of each interval $[i,i+1]$ as the integer
$i$). To be more precise, if Theorem~\ref{T-S} works for 
$n \geq N_0(k,c)$ then we apply it with $N_0(k/\eps,c\delta)$. Let
$i_0,\ldots,i_{k/\eps}$ be the intervals of this arithmetic
progression. Then, by definition each of these $k/\eps$ intervals
has a point from the set. Pick an arbitrary element of the set from
the $k$ intervals $i_0, i_{1/\eps}, i_{2/\eps},\ldots,i_{k/\eps}$. 
Then we get an $\eps$-approximate $k$-term arithmetic progression
since the distance between these intervals is at least $1/\eps$, so
the error from picking an arbitrary point in each interval is at most
$\eps$ relative to the distance between the points.  
 
It is also worth noting that our proof of Theorem~\ref{T1} 
is self contained and much simpler (from first principles) than 
the proof one gets from Szemer\'{e}di's theorem as described above. 
Moreover, the upper bound resulting from our proof is much better
than that one gets from the integer theorem. 
That is, with the quantitative bounds included, the two theorems
(\ref{T1} and~\ref{T2}) cannot be derived as corollaries of the
classical integer theorems. Indeed, as mentioned in the introduction no
combinatorial proof is known for the higher dimensional generalization
of Szemer\'{e}di's theorem due to F\"urstenberg and Katznelson.

\section{Almost collinear points} \label {S-collinear}

Let $0 <\eps<1$, and let $S$ be a finite point set in $\RR^d$.
$S$ is said to be $\eps$-{\em collinear}, if in every triangle
determined by $S$, two of its (interior) angles are at most $\eps$. 
Note that in particular, this condition implies that an
$\eps$-collinear point set is contained in a section of a cylinder
whose axis is a diameter pair of the point set, and with radius $\eps D$, 
where $D$ is the diameter; the cylinder radius is at most
$\frac{D}{2} \tan \eps \leq \eps D$, for $\eps<1$.

\begin{theorem} \label{T4} 
For any dimension $d$, positive integer $k$, and $\eps>0$, there exists
$N=N(d,k,\eps)$, such that any point set $S$ in $\RR^d$ with at least $N$ points
has a subset of $k$ points that is $\eps$-collinear.
\end{theorem} 
\begin{proof}
For simplicity, we present the proof for $d=2$; the argument for $d \geq 3$ 
is analogous.
Finitely color all the segments determined by $S$ as follows.
Choose a coordinate system, so that no two points have the same
$x$-coordinate. Put $r=\lceil \pi/\eps \rceil +1$, and let $\I$ be a 
uniform subdivision of the interval $[-\pi/2,\pi/2]$ into $r$
half-closed subintervals of length at most $\eps$. 

Let $pq$ be any segment, where $x(p) < x(q)$. Color $pq$ by $i$ if the
angle made by $pq$ with the $x$-axis belongs to the $i$th subinterval. 
Obviously this is an $r$-coloring of the segments determined by $S$.
Let $N=N(2,k,r)$, where $N(\cdot)$ is as in Theorem~\ref{T-R}. 
By Ramsey's theorem (Theorem~\ref{T-R}), for every $r$-coloring of the
segments of an $N$-element point set, there exists 
a monochromatic set $K$ of $k$ points, that is, all segments have the same
color, say $i$. Let $\Delta{pqr}$ be any triangle determined by $K$,
and assume that $x(p) < x(q) < x(r)$. Then by construction, we have 
$\angle{qpr}, \angle{prq} \leq \eps$.
This means that $K$ is $\eps$-collinear, as required..
\end{proof}

\smallskip
We conclude with an informal remark. 
Observe that the limit of an $\eps$-collinear set, when $\eps \rightarrow 0$, 
is a collinear set of points. It should be noted that one cannot hope
to find any other {\em non-collinear} pattern which is the limit
of some approximate patterns occurring in any sufficiently large point
set, no matter how large. 
Indeed, by taking all points in our ground set on a common
line, all its subsets will be collinear. 

\paragraph{Acknowledgments.}
The author thanks the anonymous reviewer of an earlier version
for the observation and the proof in the remark at the end of
Section~\ref{S-approximate}.

\end{document}